\documentclass[]{article}

\date{\today}
\usepackage[yyyymmdd]{datetime}

\usepackage{amsmath}
\usepackage{amsthm, amssymb}

\newcommand{\N}{\mathbb{N}}

\renewcommand{\rho}{\varrho}
\renewcommand{\phi}{\varphi}

\let\existstemp\exists
\let\foralltemp\forall
\renewcommand*{\exists}{\ \existstemp}
\renewcommand*{\forall}{\ \foralltemp}

\theoremstyle{plain}
\newtheorem{thm}{Theorem}%[section]

\newtheorem{cor}[thm]{Corollary}

\usepackage{hyperref}

\title{
Cover numbers by certain graph families
}
\author{M\'{a}rton Marits \thanks{\href{mailto:marton-marits-xg@ynu.jp}{marton-marits-xg@ynu.jp}} \\
\small{Yokohama National University, Yokohama, Japan}
}

\begin{document}

\maketitle

\begin{abstract}
	We define the \textit{cover number} of a graph $G$ by a graph class $\mathcal P$ as the minimum number of graphs of class $\mathcal P$ required to cover the edge set of $G$. Taking inspiration from a paper by Harary, Hsu and Miller \cite{HararyHsuMiller1977}, we find an exact formula for the cover number by the graph classes $\{ G \mid \chi(G) \leq f(\omega(G))\}$ for all non-decreasing functions $f$.
	
	After this, we establish a chain of inequalities with five cover numbers, the one by the class $\{ G \mid \chi(G) = \omega(G)\}$, by the class of perfect graphs, generalized split graphs, co-unipolar graphs and finally by bipartite graphs. We prove that at each inequality, the difference between the two sides can grow arbitrarily large. We also prove that the cover number by unipolar graphs cannot be expressed in terms of the chromatic or the clique number.
\end{abstract}


\section{Introduction}

Let $\mathcal P$ be a class of graphs. %, i.e. a set containing some graphs.
A \textit{covering} of a graph $G$ by graphs of class $\mathcal P$ is an expression $E(G) = E(G_1) \cup \dots \cup E(G_k)$ of the edges of $G$ as the union of the edge sets of the graphs $G_1, \dots G_k$, such that for all $i \in \{1, \dots, k\}$, we have $G_i \in \mathcal P$. From now, for brevity we will simply write $G_i$ in place of $E(G_i)$. Furthermore, unless stated otherwise, all logarithms in this article are taken to the base 2.

The smallest number $k$ such that $G$ admits a covering by $k$ graphs of class $\mathcal P$ is called the \textit{cover number of $G$ by $\mathcal P$}. We denote this number by $c_\mathcal{P}(G)$. Cover numbers by various classes have been investigated previously. An early survey of the topic is \cite{Harary1970}. The cover number by bipartite graphs is called \textit{biparticity} by Harary, Hsu and Miller \cite{HararyHsuMiller1977}, and their paper presents an explicit formula for its value. In Gy\'{a}rf\'{a}s, Marits and T\'{o}th \cite{GyarfasMaritsToth2024}, we have seen an investigation of the cover number by comparability graphs. The cover number where $\mathcal P$ is taken to be the class of planar graphs is called \textit{thickness} and has a rich literature, surveyed in \cite{Mutzel&al1998}. %TODO: maybe add more examples?

The following trivial general result can be obtained for cover numbers.

\begin{thm} \label{thm:triv}
	Let $\mathcal P$ and $\mathcal Q$ be two graph classes. Then $\mathcal P \subseteq \mathcal Q$ if and only if $c_{\mathcal P}(G) \geq c_{\mathcal Q}(G)$ for all graphs $G$.
\end{thm}

\begin{proof}
	The $\implies$ direction is obvious: if $\mathcal P$ is a subset of $\mathcal Q$, then any covering by $\mathcal P$-graphs is also a covering by $\mathcal Q$-graphs.
	
	To prove the converse, suppose that $c_{\mathcal P}(G) \geq c_{\mathcal Q}(G)$ for all graphs $G$. Then if $P \in \mathcal P$, then obviously $c_{\mathcal P}(P) = 1$. The inequality we assumed therefore implies $c_{\mathcal Q}(P) = 1$, which in turn means that $P \in \mathcal Q$.
\end{proof}

The cover number by bipartite graphs is named \textit{biparticity} by Harary, Hsu and Miller in \cite{HararyHsuMiller1977}, and is shown to be equal to $\left\lceil \log_2 \chi(G) \right\rceil$ for all graphs $G$. In this article, we will denote this number by $c_{\mathbf{BIP}}$. Generalizing their result, it is also possible to find a formula for the cover number by graphs of bounded chromatic numbers. The following result is thus obtained for the class $%\mathcal P = 
\{G \mid \chi(G) \leq k\}$. We will denote the cover number by this class by $c_{\{\chi \leq k\}}$.

\begin{thm}[\cite{HararyHsuMiller1977}] \label{thm:harary}
	For all graphs $G$, we have\[
		c_{\{\chi \leq k\}}(G) = \left\lceil \frac{\log \chi(G)}{\log k} \right\rceil
	\]
\end{thm}

The theorem and the proof only work for constant $k$. In this paper, we will show that substituting a function of $\omega(G)$ for $k$ still yields the expected formula.

Let $f\colon \N \to \N$ be a function that is non-decreasing, and one that majorizes the identity function. We say that a graph class $\mathcal P$ is \textit{$\chi$-bounded} with \textit{$\chi$-binding function} $f$ if every graph $G \in \mathcal P$ has $\chi(G) \leq f(\omega(G))$. This concept was introduced by Gyárfás in \cite{Gyarfas1987}, and a survey of the topic can be found in \cite{ScottSeymour2020}.

Using the idea of the above definition, for any suitable $f$ we can define the graph class $\{ G \mid \chi(G) \leq f(\omega(G))\}$, which we will abbreviate as $\{\chi \leq f(\omega)\}$. Generalizing \ref{thm:harary}, we prove the following statement about the cover number by this class of graphs.

\begin{thm} \label{thm:chibound}
	For all graphs $G$, we have \[
		c_{\{\chi \leq f(\omega)\}}(G) = \left\lceil \frac{\log \chi(G)}{\log f(\omega(G))} \right\rceil
	\]
\end{thm}

In particular, taking $f$ to be the identity function, we get an exact expression for the covering number by graphs $G$ with $\chi(G) = \omega(G)$.

\begin{cor} \label{thm:chiomega}
	For all graphs $G$, we have \[
		c_{\{\chi  = \omega\}}(G) = \left\lceil \frac{\log \chi(G)}{\log \omega(G)} \right\rceil 
	\]
\end{cor}

Taking $f$ in Theorem \ref{thm:chibound} to be a constant function formally gives back Theorem \ref{thm:harary}. Our result is therefore an extension of this previously-known theorem. %FORMALLY... technically I'm not allowed to put a function that doesn't majorize the identity function there

%It is important to mention that in this theorem, we get an exact formula for the class of graphs with $\chi(G) = \omega(G)$, a class that is slightly larger than the set of perfect graphs.

%TODO: definitions for perfect graph, GSP graph, comparability graph

Obviously, a graph class $\mathcal P$ is chi-bounded with chi-binding function $f$ if and only if $\mathcal P \subseteq \{\chi \leq f(\omega)\}$. Using the trivial Theorem \ref{thm:triv} and the new Theorem \ref{thm:chibound}, we can give an equivalent condition for chi-boundedness.

\begin{cor}
	The graph class $\mathcal P$ is chi-bounded with chi-binding function $f$ if and only if \[
		c_{\mathcal P}(G) \geq \left\lceil \frac{\log \chi(G)}{\log f(\omega(G))} \right\rceil
	\] for all graphs $G$.
\end{cor}

A graph $G$ is called \textit{perfect} if $G$ all of its induced subgraphs are in the class $\{\chi = \omega\}$. This class of graphs has garnered considerable interest, much more than the class $\{\chi = \omega\}$ that we found results about. We will denote the cover number by perfect graphs by $c_{\mathbf{PERF}}$.

A \textit{unipolar graph} is a graph $G$ such that $V(G) = V(A) \dot{\cup} V(B)$ such that $A$ is a clique and $B$ is the disjoint union of cliques. A graph is called \textit{co-unipolar} if its complement is unipolar. Furthermore, a graph is called \textit{generalized split} (GSP) if it is either unipolar or co-unipolar. Generalized split graphs make up ``almost all" perfect graphs, according to a result of Pr\"{o}mel and Steger \cite{PromelSteger1992} and the Strong Perfect Graph theorem\cite{Chudnovsky&al2006}. We will denote the cover number by unipolar graphs by $c_{\mathbf{UNIP}}$, the cover number by co-unipolar graphs as $c_{\mathbf{coUNIP}}$ and the cover number by GSP graphs as $c_{\mathbf{GSP}}$. %and as such we expect that the cover number by GSP graphs, denoted $c_{\mathbf{GSP}}$ is close to $c_{\mathbf{PERF}}$. %We show that this is not the case, in fact, $c_{\mathbf{PERF}}$ and $c_{\mathbf{GSP}}$ can be arbitrarily far away from each other.

With all these cover numbers defined, we can repeatedly use Theorem \ref{thm:triv} to establish the following chain of inequalities. In order to make the equations look cleaner, we take each cover number, as well as the chromatic and clique numbers as functions acting on graphs. Between functions, the inequalities are defined pointwise.

\begin{thm} \label{thm:chain}
	Taking each cover number, as well as the clique number and the chromatic number as functions acting on graphs, the following holds:\[
		\left\lceil \frac{\log \chi}{\log \omega} \right\rceil = c_{\{\chi = \omega\}} \leq c_{\mathbf{PERF}} \leq c_{\mathbf{GSP}} \leq c_{\mathbf{coUNIP}} \leq c_{\mathbf{BIP}} = \left\lceil \log \chi \right\rceil
	\]
\end{thm}

%As for the cover number by perfect graphs, $c_{\mathbf{PERF}}$, I have not found an explicit formula in terms of $\chi$ and $\omega$, but we can trivially say that it is lower bounded by the expression $\left\lceil \frac{\log \chi(G)}{\log \omega(G)} \right\rceil$. It can, however, be arbirarily larger than this value. In other words, the following theorem holds.

For the cover numbers in the middle, we %I think `I' would be better here
have found no explicit formula in terms of the chromatic number and the clique number. We can however establish that at every inequality, the difference between the two cover numbers can get arbitrarily large. This is despite the close relationship between the graph classes in question. For example, ``almost all" perfect graphs are GSP \cite{PromelSteger1992}, but there are perfect graphs such that require many GSP graphs for a covering.

Thus, the following four theorems can be proved. %proven?

\begin{thm} \label{thm:far}
	Let $0 < k \leq \ell$ be two natural numbers. Then there exists a graph $G_{k,\ell}$ such that $c_{\{\omega = \chi\}}(G_{k,\ell}) = k$ and $c_{\mathbf{PERF}}(G_{k,\ell}) \geq \ell$.
\end{thm}

\begin{thm} \label{thm:far2}
	For all $k \geq 1$, there exist perfect graphs $H_k$ such that $c_{\mathbf{GSP}}(H_k) \geq \frac k 2$.
\end{thm}

\begin{thm} \label{thm:far3}
	Let $0 < k \leq \ell$ be two natural numbers. Then there exist GSP graphs $A_{k,\ell}$ such that $c_{\mathbf{coUNIP}}(A_{k,\ell}) = \min\{k, \log \ell\}$.
\end{thm}

\begin{thm} \label{thm:far4}
	Let $k \geq 1$. Then there exist co-unipolar graphs $B_k$ such that $c_{\mathbf{BIP}}(B_k) = \log k$.
\end{thm}

%These observations imply significant restrictions on the 

Another observation is made about the cover number by unipolar graphs, $c_{\mathbf{UNIP}}$. Since unipolar graphs do not contain all bipartite graphs, their cover number is not bounded from above by $\lceil \log \chi \rceil$. We show that this cover number can not be bounded above by any other function of $\chi$ and $\omega$ either, by showing an example of a family of graphs with constant chromatic and clique numbers, but unbounded $c_{\mathbf{UNIP}}$.

\begin{thm} \label{thm:uniplarge}
	For all $d \geq 8$, there exist bipartite graphs $Q_d$ with $c_{\mathbf{UNIP}}(Q_d) = d$.
\end{thm}

%Since bipartite graphs have $\omega(G) = \chi(G) = 2$, this theorem means that the cover number by unipolar graphs cannot be expressed by the chromatic number, indeed, the cover number can be arbitrarily large even for graphs with $\chi(G) = 2$.

\section{Proofs}

\begin{proof}[Proof of Theorem \ref{thm:chibound}]
	We are to prove that \[
		c_{\{\chi \leq f(\omega)\}} = \left\lceil \frac{\log \chi(G)}{\log f(\omega(G))} \right\rceil
	\]
	
	Let us first prove the $\geq$ direction. In this case, we can safely remove the integer part sign and rearrange the equation into\[
		\chi(G) \leq f(\omega)^{c_{\{\chi \leq f(\omega)\}}}
	\]
	
	Thus, to prove the $\geq$ direction, we merely need to show a coloring of any graph $G$ by $f(\omega)^{c_{\{\chi \leq f(\omega)\}}}$ colors. Take any optimal covering of $G$ by $t = c_{\{\chi \leq f(\omega)\}}(G)$ graphs $G = G_1 \cup \dots \cup G_t$ where $\chi(G_i) \leq f(\omega(G_i))$ for all $G_i$'s. Note that each $G_i$ has $\omega(G_i) \leq \omega(G)$.
	
	Take an optimal coloring of each of the covering graphs $G_i$. We can now color the original graph $G$ by coloring each vertex with a $t$-tuple containing the colors of the vertex by the optimal colorings of the $G_i$'s. This results in a proper coloring, since any edge $e = xy$ of $G$ is covered by at least one $G_i$, resulting in at least one coordinate $i$ where the colors of $x$ and $y$ differ. In total, this is a coloring by at most\[
		\prod_{i = 1}^{t} \chi(G_i) \leq \prod_{i = 1}^{t} f(\omega(G_i)) \leq \prod_{i = 1}^{t} f(\omega(G)) = f(\omega(G))^t
	\] colors, proving the inequality.
	
	To prove the opposite inequality, it suffices to show a covering of $G$ with $\lceil\log_{f(\omega(G))}\chi(G)\rceil$-many graphs with $\chi \leq f(\omega)$. As a first step, take an optimal coloring of the vertices of $G$ with $\chi(G)$ colors. %Relabel the colors such that they are multisets consisting of elements of $A = \{1, \dots, \log_{\omega} \chi\}$, each element appearing at most $\omega - 1$ times.
	Relabel the colors such that each color is a function $\phi\colon A = \{1, \dots, \lceil\log_{f(\omega(G))}\chi(G)\rceil\} \to \{1, \dots, f(\omega(G))\}$. There are enough such functions to make this possible. Let $\Omega$ be a maximal clique in $G$. We can assume without loss of generality that the vertices of $\Omega$ are colored by some constant functions, as $f(\omega(G)) \geq \omega(G)$, since we assumed that $f$ majorizes the identity function. For each vertex $x \in V(G)$, denote its color as a function by $\phi_x$.
		
	Now for each element $i \in A$, let $G_i$ consist of edges $e = xy$ where $\phi_x(i) \neq \phi_y(i)$. Then all edges are covered by at least one $G_i$, since the colors of adjacent vertices necessarily differ as functions. Since $\Omega$ was colored by the constant functions, each $G_i$ contains all edges of $\Omega$, therefore $\omega(G_i) = \omega(G)$. Furthermore, each $G_i$ can be colored with $f(\omega(G))$ colors by using the value of $\phi(i)$ at each vertex. Therefore the $G_i$'s form the desired covering of $G$ by $\chi \leq f(\omega)$ graphs.
\end{proof}
	
\begin{proof}[Proof of Theorem \ref{thm:far}]
	Let $0 < k \leq \ell$ be the desired cover numbers. Let us first construct a graph $G$ in the following way. Let $Z_{2^\ell}$ be a graph with $\omega(Z_{2^\ell}) = 2$ and $\chi(Z_\ell) = {2^\ell}$. Such graphs exist, for example we can use the famous construction of Zykov \cite{Zykov1949} to obtain $Z_{2^\ell}$. We then take the complete graph $K_{2^\ell}$ and have $G = Z_{2^\ell} \cup K_{2^\ell}$ as a disjoint union. $G$ now has $\chi(G) = \omega(G) = 2^\ell$, but if one were to cover $G$ with perfect graphs as $G = \bigcup P_i$, then this would induce a covering of $Z_{2^\ell}$ by perfect graphs. Due to the trivial inequality $c_{\mathbf{PERF}} \geq c_{\{\chi = \omega\}}$ derived from Theorem \ref{thm:triv}, and Corollary \ref{thm:chiomega}, we can thus conclude that this covering of $Z_{2^\ell}$ needs at least $\log {2^\ell} = \ell$ perfect graphs.
	
	If $k = 2$, then we are done. Otherwise take a graph $R_k$ which has $c_{\{\chi = \omega\}}(R_k) = k$. This is possible, since all we need is a graph with $\chi(R_k) = \omega(R_k)^k$. Then the disjoint union $Z_{2^\ell} \cup K_{2^\ell} \cup R_k$ has exactly the desired cover numbers.
\end{proof}

\begin{proof}[Proof of Theorem \ref{thm:far2}]
	A graph is called a \textit{comparability graph} if it can be obtained from a partially ordered set $(P, \preceq)$ by taking $P$ as the vertex set and connecting two vertices if they are comparable by $\preceq$. It is known that all comparability graphs are perfect, this result is easy to derive from Mirsky's theorem \cite{Mirsky1971}. Let us write the cover number by comparability graphs of a graph $G$ as $c_{\mathbf{COMP}}(G)$.

	We use the construction from \cite[Theorem 8]{GyarfasMaritsToth2024}. In that paper, it is proven that there exist perfect graphs with arbitrarily large cover number by comparability graphs. In the same paper, it is also proved \cite[Theorem 4]{GyarfasMaritsToth2024} that all GSP graphs have $c_{\mathbf{COMP}} \leq 2$.
	
	Let $H_k$ thus be a graph that is perfect and $c_{\mathbf{COMP}}(H_k) = k$. Then trivially $c_{\mathbf{PERF}}(H_k) = 1$. If we then have a covering $H_k = G_1 \cup \dots \cup G_c$ by GSP graphs for some $c$, then since each of the $G_i$'s are covered by two comparability graphs, we have a covering of $H_k$ by $2c$ comparability graphs, meaning that $2c \geq k$. Thus, we must have $c_{\mathbf{GSP}}(H_k) \geq \frac{k}{2}$. 
\end{proof}

\begin{proof}[Proof of Theorem \ref{thm:far3}]
	Let $A_{k,\ell}$ be a graph obtained as a disjoint union of $k$ complete graphs $K_\ell$. These are then unipolar, and therefore GSP.
	
	Notice that any co-unipolar graph $G = A \cup B$ has the following properties. We have $A$ an independent set and $B$ a complete partite graph. If $B$ consists of a single partite graph (i.e. its complement is a clique) then $G$ is bipartite. Otherwise $B$ has internal edges, meaning that the edge set of $G$ is connected.
	
	%TODO !!!!!!
	Thus, to cover $A_{k,\ell}$ by co-unipolar graphs, %one of two strategies can be employed. If we always use non-bipartite co-unipolars, we need as many graphs to cover $G_{k,l}$ as it has components, i.e. $k$-many. If we use bipartite graphs instead, the result from \cite{HararyHsuMiller1977} implies that we need $\log l$-many graphs for a covering. Mixing the two strategies
	we can use the two types of co-unipolar graphs described above. Using bipartite graphs needs at least $\log \ell$ components to cover any single component of $A_{k,\ell}$, as per the results from \cite{HararyHsuMiller1977}. Using the non-bipartite type of co-unipolar graphs needs at least one covering graph per component, in total needing at least $k$, which is enough. It is easy to see that any `mixed' strategy using both types of co-unipolar graphs will fail to cover $A_{k,\ell}$ with less than $\min\{\log \ell, k\}$ graphs. %TODO
	
	Thus, $A_{k,\ell}$ is a graph with $c_{\mathbf{GSP}}(A_{k,\ell}) = 1$ and $c_{\mathbf{coUNIP}}(A_{k,\ell}) = \min\{\log \ell, k\}$.
\end{proof}

\begin{proof}[Proof of Theorem \ref{thm:far4}]
	Since for any graph $G$, $c_{\mathbf{BIP}}(G) = \lceil \log \chi(G) \rceil$ as per the result of Harary, Hsu and Miller \cite{HararyHsuMiller1977}, we can take $B_k$ to be the complete graphs. Complete graphs are co-unipolar, as $K_k = A \cup B$ where $A = \emptyset$ and $B$ is a $k$-partite graph where each partite set is a single vertex.
\end{proof}

\begin{proof}[Proof of Theorem \ref{thm:uniplarge}]
	The desired graphs $Q_d$ are the hypercube graphs defined by $V(Q_d) = \{0, 1\}^d$, with edges between vertices that have exactly one differing coordinate. Indeed, these graphs are bipartite, where the bipartition is based on the parity of the number of $1$'s in each vertex.
	
	A potential covering of $Q_d$ by $d$ unipolar graphs is the following. Let $G_i$ consist of the edges that connect vertices that differ in the $i$'th coordinate. Each of these $G_i$'s now consist of $2^{d-1}$ independent cliques, forming a unipolar graph. This defines a covering -- in fact, a partition -- of $Q_d$ by $d$ unipolar graphs.
	
	To show that for large values of $d$, there is no better covering, let us consider the largest possible unipolar subgraph of $Q_d$. Since $Q_d$ only has cliques of size 2, the $A$ part of any unipolar subgraph of $Q_d$ has to be a single edge. The $B$ part can then be at most $2^{d-1} - 1$ further cliques in some arrangement. The unipolar graph can also have edges connecting $A$ and $B$, but the amount of these edges can at most be the number of edges incident to the one edge covered by $A$, of which there are $2(d-1)$. Summarizing, a unipolar subgraph of $Q_d$ can have at most $2^{d-1} + 2(d-1)$ edges.
	
	Therefore, to cover the entirety of $Q_d$, which has $d\cdot 2^{d-1}$ edges, at least \[
		\frac{d \cdot 2^{d-1}}{2^{d-1} + 2(d-1)} = d - \frac{2d(d-1)}{2^{d-1} + 2(d-1)}
	\] unipolar graphs are needed. This expression obviously tends to $d$ as $d$ goes to infinity, and indeed, for $d \geq 8$, its upper integer part will be $d$.
	
	Thus for $d \geq 8$, $c_{\mathbf{UNIP}(Q_d)} = d$, proving the theorem.
\end{proof}

%\section{Conclusions} ?

\section{Open problems}

The following open problems remain:

\begin{enumerate}
	\item Determining the exact value of $c_{\mathbf{PERF}}$, $c_{\mathbf{coUNIP}}$ and $c_{\mathbf{GSP}}$ in terms of the chromatic number or otherwise.
	\item Determining the exact value of $c_{\mathbf{UNIP}}$. We have shown that in terms of the chromatic number, this value is unbounded, but a formula using different graph parameters may still be attainable.
	\item How to determine whether a given a function $f$ that assigns a number to each graph is a cover number by some class $\mathcal P$?
\end{enumerate}

\section{Acknowledgements}

I would like to thank my supervisor, Kenta Ozeki %flip name?
for his numerous useful comments and suggestions.

%\section{References}

\end{document}